\documentclass[12pt]{amsart}
\usepackage{amsmath,amsthm,amscd,amsfonts, amssymb}
\newcommand{\sect}[1]{\section{#1}\setcounter{equation}{0}}
\newcommand{\subsect}[1]{\subsection{#1}}

\font\mbn=msbm10 scaled \magstep1
\font\mbs=msbm7 scaled \magstep1
\font\mbss=msbm5 scaled \magstep1
\newfam\mbff
\textfont\mbff=\mbn
\scriptfont\mbff=\mbs
\scriptscriptfont\mbff=\mbss\def\mbf{\fam\mbff}

\def\Co{{\mbf C}}

\newtheorem{Th}{Theorem}[section]
\newtheorem{Lm}[Th]{Lemma}

\newtheorem{D}[Th]{Definition}
\newtheorem{Proposition}[Th]{Proposition}

\newtheorem*{theorem}{Theorem 1}
\newtheorem*{theo}{Theorem 2}
\begin{document}

\title[Composition Conditions for Classes of Analytic Functions]{Composition Conditions for Classes of Analytic Functions}

\author{Alexander Brudnyi} 
\address{Department of Mathematics and Statistics\newline
\hspace*{1em} University of Calgary\newline
\hspace*{1em} Calgary, Alberta\newline
\hspace*{1em} T2N 1N4}
\email{albru@math.ucalgary.ca}

\thanks{Research supported in part by NSERC}

\begin{abstract}
We prove that for classes of analytic functions tree composition condition and composition condition coincide.
\end{abstract}

\date{}

\maketitle

\sect{Main Results}
\subsection{} In the paper we continue to study interrelation between composition condition (CC) and tree composition condition (TCC) for certain classes of continuous functions started in \cite{BY}. These conditions are defined as follows.
\begin{D}\label{def1}
A family $\{f_\alpha\}_{\alpha\in\Lambda}$ of complex continuous functions on an interval $[a,b]\subset\mathbb R$ is said to satisfy a tree composition condition if there exist a tree $T$, a continuous map $h : [a, b]\to T$ with $h(a)=h(b)$ and a family of continuous functions $\{\tilde f_\alpha\}_{\alpha\in\Lambda}$ on $T$ such that
\[
f_\alpha(x)=\tilde f_\alpha(h(x)),\qquad x\in [a,b],\quad \alpha\in\Lambda.
\]

If $T\subset\mathbb R$ the family is said to satisfy a composition condition.
\end{D}
CC and TCC recently appeared in various questions related to the, so-called, center problem for ODEs asking whether the differential equation 
\begin{equation}\label{ode}
v'=\sum_{j=1}^\infty a_j v^j,\qquad a_j\in L^\infty([a,b]),\quad j\in\mathbb N,
\end{equation}
determines a center, i.e., for all sufficiently small initial values the corresponding solutions satisfy $v(a)=v(b)$. In particular, it was proved in \cite{B1}, \cite{B3} that if all finite subsets of the family $\{\int_a^x a_j(s)ds\}_{j\in\mathbb N}$ satisfy TCC, then equation \eqref{ode} determines a center, and, in a certain statistical sense, almost all centers of equations \eqref{ode} are obtained in this way (see also  \cite{AL}, \cite{BlRY}, \cite{BFY}, \cite{BRY}, \cite{B2}, \cite{B4}, \cite{B5}, \cite{C}, \cite{CL}, \cite{FPYZ}, \cite{MP}, \cite{P1}, \cite{PRY}, \cite{P2}, \cite{Z} for other applications of CC). 

For a long time it was unclear why the more stronger CC appears more often than the topologically more adequate TCC. A partial explanation of this fact was given in \cite{BY}: it was shown that CC is indeed stronger
than TCC in general, e.g., for real piecewise-linear functions, but coincides with
it for real entire or rational functions. In this paper we provide a complete explanation of this phenomenon by showing that CC and TCC coincide for classes of real analytic functions (and this is the case of most of known applications of CC in analysis).

\subsection{Vanishing of Iterated Integrals and Composition Conditions}
In this part we recall some results proved in \cite{B1} and \cite{B2}.

Suppose functions $f_j: [a,b]\to\mathbb C$ are Lipschitz with $f_j(a)=0$; by $f_j'\in L^\infty([a,b])$ we denote their derivatives ($1\le j\le n$). Consider a Lipschitz map 
\begin{equation}\label{map}
F=(f_1,\dots, f_n): [a,b]\to\mathbb C^n.
\end{equation}
Let $\widehat\Gamma$ be the polynomially convex hull of $\Gamma:=F([a,b])$, i.e., the set of points $z\in\mathbb C^{n}$ such that if $p$ is any holomorphic polynomial in $n$ variables, 
\[
|p(z)|\leq\max_{x\in \Gamma}|p(x)|.
\]
Then $\widehat\Gamma$ is compact connected and $\widehat\Gamma\setminus\Gamma$ is either empty or a purely one-dimensional complex analytic subvariety of $\mathbb C^n\setminus\Gamma$, see \cite{A}.

In what follows we assume that 
\begin{equation}\label{hull}
\widehat\Gamma=\Gamma.
\end{equation}
This is fulfilled if, e.g.,  $\Gamma$ is a subset of a compact set $K$ in a $C^{1}$ manifold  with no complex tangents such that $\widehat K=K$, see \cite[Thm.~17.1]{AW}. For instance, one can take as such $K$ any compact subset of $\mathbb R^{n}$. 

Also, we assume that 
\begin{equation}\label{triang}
\Gamma\ \text{is Lipschitz triangulable}.
\end{equation}
This means that there exist Lipschitz embeddings $g_{j}:[0,1]\to\mathbb C^{n}$ such that $\Gamma=\cup_{j=1}^{s}\,g_{j}([0,1])$, for $i\neq j$ the intersection $g_{i}([0,1])\cap g_{j}([0,1])$ consists of at most one point and
the inverse maps $g_{j}^{-1}:g_{j}([0,1])\to\mathbb R$ are locally Lipschitz on $g_{j}((0,1))$. In particular, such $\Gamma$ is arcwise connected and locally simply connected and its fundamental group $\pi_1(\Gamma)$ is a free group with finite number of generators.

For example, $\Gamma$ is Lipschitz triangulable if the map $F$ is non-constant analytic.

The next result was established in \cite[Cor.~1.12, Thm.~1.14]{B1}.
\begin{theorem}
The following conditions are equivalent:
\begin{itemize}
\item[(1a)]
The path $F:[a,b]\to\Gamma$ is closed (i.e., $F(a)=F(b)$) and represents the unit element of $\pi_1(\Gamma)$ (i.e., it is contractible in $\Gamma$);
\item[(1b)]
Functions $f_1,\dots, f_n$ satisfy TCC;
\item[(1c)]
For all possible $i_1,\dots, i_k\in\{1,\dots, n\}$ iterated integrals
\[
\int\cdots\int_{a\leq s_{1}\leq\cdots\leq s_{k}\leq b}f'_{i_{k}}(s_{k})\cdots f'_{i_{1}}(s_{1})\ \!ds_{k}\cdots ds_{1}
\]
are equal to zero.
\end{itemize}
\end{theorem}

Weaker homotopic and composition conditions are given by vanishing of the first order moments of $F$, the iterated integrals of the simplest form, see \cite[Cor.~3.11]{B2}.
\begin{theo}
The following conditions are equivalent:
\begin{itemize}
\item[(2a)]
The path $F:[a,b]\to\Gamma$ is closed and represents an element of the commutator subgroup $[\pi_1(\Gamma),\pi_1(\Gamma)]\subset\pi_1(\Gamma)$;
\item[(2b)]
There exists a continuous map $h: [a,b]\to\Gamma_{A}$, $h(a)=h(b)$, where $\Gamma_{A}\stackrel{r}{\rightarrow}\Gamma$ is the regular covering of $\Gamma$ with the deck transformation group \penalty-10000 $\pi_1(\Gamma)/[\pi_1(\Gamma),\pi_1(\Gamma)]\, (=:H_1(\Gamma))$, such that $F=r\circ h$;
\item[(2c)]
For all possible $n_1,\dots , n_k\in\mathbb N$ and $i_1,\dots , i_{k+1}\in \{1,\dots, n\}$
first order moments 
\[ 
\int_a^b(f_{i_1}(s))^{n_1}\cdots (f_{i_k}(s))^{n_k}f'_{i_{k+1}}(s)ds
\]
are equal to zero.
\end{itemize}
\end{theo}

\subsection{Composition Conditions for Analytic Functions}
One can easily show that conditions of Theorems 1 and 2 are not equivalent for generic $f_1,\dots, f_n$. Strikingly, they are equivalent if the functions are analytic.

\begin{Th}\label{compos}
Let ${\mathcal F}$ be a family of complex analytic functions defined in open neighbourhoods of a compact interval $[a,b]$ whose finite subsets satisfy conditions \eqref{hull} and \eqref{triang} . Then ${\mathcal F}$ satisfies CC if and only if each of its finite subsets satisfies one of the conditions (1a)-(1c) or (2a)-(2c).
\end{Th}

The proof of the result is based on the following

\begin{Th}\label{te1}
Let $F: U\to\mathbb C^n$ be a nonconstant holomorphic map defined in a domain $U$ containing a compact interval $I\subset\mathbb R$ such that $F(I)\subset\mathbb R^n$. 
There exist nonconstant holomorphic maps $h: D\to\mathbb C$ of a domain $D\Subset U$ containing $I$ and $G:h(D)\to\mathbb C^n$ such that
\begin{itemize}
\item[(1)]
$F|_D=G\circ h$;
\item[(2)]
$h(I)\subset\mathbb S:=\{z\in\mathbb C\, ;\, |z|=1\}$;
\item[(3)]
$G$ is injective outside a finite subset of $h(D)$.
\end{itemize}
\medskip

If also $F|_{D_1}=G_1\circ h_1$ for some holomorphic maps $h_1:D_1\to\mathbb C$ of a domain $D_1\Subset U$ containing $I$ and $G_1:h_1(D_1)\to\mathbb C^n$, then there exists a holomorphic map $q: h_1(D\cap D_1)\to h(D)$ such that $h=q\circ h_1$ on $D\cap D_1$ and
$G_1=G\circ q$ on $h_1(D\cap D_1)$.
\end{Th}

\sect{Proof of Theorem \ref{te1}}
\subsect{} In this subsection we prove some preparatory results and in the next one apply them to prove the theorem.

In what follows for a set $K\subset\mathbb C^k$ by ${\rm cl}(K)$ and $\partial K$ we denote its closure and boundary. By $\,\bar{\,}:\mathbb C^k\to\mathbb C^k$ we denote the involution mapping each coordinate $z_i$ on $\Co^k$ to its conjugate $\bar{z_i}$, $1\le i\le k$.

\begin{Proposition}\label{pr1}
Suppose $F:=(f_1,\dots, f_n): U\to\mathbb C^n$ is a nonconstant holomorphic map defined in an open neighbourhood $U\subset\mathbb C$ of $[a,b]\Subset\mathbb R$. There exist an irreducible one-dimensional complex space $X\subset\mathbb C^n$ and a domain $D\Subset U$ invariant under conjugation and containing $[a,b]$ such that $F(D)\Subset X$. If, in addition, all $f_j|_{U\cap \mathbb R}$ are real-valued, then $\bar{X}=X$.
\end{Proposition}
\begin{proof}
Since the set of critical points of $F$ is discrete, we can find a domain $E\Subset U$ containing $[a,b]$ such that $\bar{E}=E$ and $\partial E$ is analytic and does not contain critical points of $F$. Then $\Gamma:=F(\partial E)\subset\mathbb C^n$ is a closed analytic curve whose singular points may be only points of self-intersection. By the definition $S:=F^{-1}(\Gamma)\subset U$ is an analytic subset of real dimension $1$. Thus, $E\setminus S$ is disjoint union of domains $V_1,\dots, V_s$. By the Remmert proper mapping theorem, see \cite{RS}, $Y:=\cup_{j} Y_j$, $Y_j:=F(V_j)$, $1\le j\le s$, is a one-dimensional complex subvariety of $\mathbb C^n\setminus\Gamma$ and $Y_j$ are its irreducible components.  

Let $F\bigl(S\cap [a,b]\bigr):=\{x_1,\dots, x_p\}$. We set 
$$
F^{-1}(x_j)\cap {\rm cl}(E):=\{z_{j1},\dots, z_{jk_j}\},\qquad 1\le j\le p.  
$$

Without loss of generality we may assume that $f_1$ is nonconstant. Then there exist a number $r>0$ and, for each $1\le\ell\le k_j$, $1\le j\le p$, open disks $D_{r}(f_1(z_{j1}))\Subset\mathbb C$ centered at $f_1(z_{j1})$ of radius $r$ and open simply connected neighbourhoods $U_{j\ell}\Subset U$ of $z_{j\ell}$ with analytic boundaries such that
\begin{itemize}
\item[(1)] All $U_{j\ell}$ are mutually disjoint;
\item[(2)] There exists a holomorphic coordinate $w_{j\ell}$ on $U_{j\ell}$ such that 
$w_{j\ell}(z_{j\ell})=0$, 
$f_1|_{U_{j\ell}}=f_1(z_{j1})+w_{j\ell}^{d_{j\ell}}$ for some $d_{j\ell}\in\mathbb N$ and $U_{j\ell}:=\{z\in\mathbb C\, ,\, |w_{j\ell}(z)|< r^{1/d_{j\ell}}\}$ (hence, $f_1|_{U_{j\ell}}$ maps $U_{j\ell}$ properly onto $D_{r_j}(f_1(z_{j1}))$).
\end{itemize}
Diminishing all $U_{j\ell}$, if necessary, we may assume in addition that 
\begin{itemize}
\item[(3)] 
There exists a domain $U'\Subset U$ such that

${\rm cl}(E)\cup\bigl(\cup_{j,\ell}\,U_{j\ell}\bigr)\Subset U'$ and $F^{-1}\bigl(F\bigl(\cup_{j,\ell}\,U_{j\ell}\bigr)\bigr)\cap U'=\cup_{j,\ell}\,U_{j\ell}$;
\end{itemize}
\begin{itemize}
\item[(4)]
If $z_{j\ell}\in E$, then $F|_{U_{j\ell}}: U_{j\ell}\to\mathbb C^n$ is a (nonproper) embedding (recall that $\partial E$ does not contain critical points of $F$);
\item[(5)]
If $z_{j\ell}\not\in \partial E$, then $U_{j\ell}\Subset E$.
\end{itemize}
If $F$ maps $U\cap\mathbb R$ into $\mathbb R^n$, then we may choose $U_{j\ell}$ so that
\begin{itemize}
\item[(6)]
The set $\sqcup_{j,\ell}\,U_{j\ell}$ is invariant under conjugation.
\end{itemize}

Property (3) implies that there exist open neighbourhoods $N_j$ of points $x_j\in\mathbb C^n$, $1\le j\le p$, such that
$$
F^{-1}(N_j)\cap U'\Subset \cup_{1\le\ell\le k_j} U_{j\ell}.
$$
Thus we can find some $0<r_*<r$ such that $U_{j\ell}':=\{z\in U_{j\ell}\, ,\, |w_{j\ell}(z)|< r_*^{1/d_{j\ell}}\}\Subset
F^{-1}(N_j)$, $1\le\ell\le k_j$, and if $F|_{U\cap \mathbb R}$ is real-valued, then the set $\sqcup_{j,\ell}\,U_{j\ell}'$ is invariant under conjugation. (Also, $f_1|_{U_{j\ell}'}$ maps $U_{j\ell}'$ properly onto $D_{r_*}(f_1(z_{j1}))$.)

We set
$$
C:=\cup_{j,\ell}\,F(\partial U_{j\ell}').
$$
Then 
$$
F^{-1}(C)\cap U'=\cup_{j,\ell}\,\partial U_{j\ell}'.
$$

Indeed, by our definition $F^{-1}(C)\cap U'\subset F^{-1}(\cup_j N_j)\cap U'$. Thus each $z\in F^{-1}(C)$ belongs to some $U_{j\ell}$. If $z\not\in \partial U_{j\ell}'$, then $|f_1(z)-f_1(z_{j1})|\ne r_*$. However, since $F(z)\in C$, we must have $|f_1(z)-f_1(z_{j1})|= r_*$, a contradiction showing that $z\in \partial U_{j\ell}'$, as required. 

Now, by the proper mapping theorem $Z':=\cup_{j,\ell}\,Z_{j\ell}'$, $Z_{j\ell}':=F(U_{j\ell}')$, is a one-dimensional complex subvariety of $\mathbb C^n\setminus C$ with irreducible components $Z_{j\ell}'$.

We set $Z_{j\ell}:=F(U_{j\ell}'\cap E)$. According to properties (4) and (5), $Z_{j\ell}$ is an open subset of $Z_{j\ell}'$. 
Next, we define $Z:=\cup_{j,\ell}\,Z_{j\ell}$ and
$$
W:=Y\cup Z
$$
and show that $W\subset\mathbb C^n$ is a one-dimensional complex space.

First, suppose that $z\in Y$. Since the latter is a complex subvariety of $\mathbb C^n\setminus\Gamma$, there exists an open neighbourhood $V$ of $z$ in $\mathbb C^n\setminus\Gamma$ and a family of holomorphic functions $g_1,\dots, g_p$ on $V$ such that
$$
Y\cap V=\{v\in V\, ;\, g_1(v)=\cdots=g_p(v)=0\}.
$$
If $Y\cap V=W\cap V$, then $W\cap V$ is a complex subvariety of  $V$. For otherwise each
open neighbourhood $O\Subset V$ of $z$ contains points of $Z\setminus Y$. Choose such neighbourhoods $\{O_i\}_{i\ge 1}$ so that $\cap_{i=1}^\infty O_i=\{z\}$ and let 
$x_i\in (Z\setminus Y)\cap O_i$, $i\ge 1$. Passing to a subsequence of the sequence $\{x_i\}$, if necessary, we may assume without loss of generality that there exist indices $j,\ell$ and points $y_i\in U_{j\ell}'\cap E$ such that $F(y_i)=x_i$, $i\ge 1$, and $\{y_i\}_{i\ge 1}$ converges to a point $y\in {\rm cl}(U_{j\ell}')\cap {\rm cl}(E)$. By definition $F(y)=z\in Y\cap V$; hence, $y\in E\setminus S$.  Then, by the continuity of $F$, since $Y$ is relatively open in $\mathbb C^n\setminus\Gamma$,  there exists an open neighbourhood $O\subset E\setminus S$ of $y$ such that $F(O)\subset Y\cap V$. In particular, pullbacks of functions $g_1,\dots, g_p$ to $O$ by $F$ vanish on $O$. But then for all $i\ge i_0$ (for some $i_0\in\mathbb N$) these pulled back functions vanish at $y_i$ as well, that is, $g_1,\dots, g_p$ vanish at points $x_i$ with $i\ge i_0$. This shows that such $x_i\in Y$, a contradiction proving that there exists an open neighbourhood $O\Subset V$ of $z$ such that $Y\cap O=W\cap O$, i.e., $W\cap O$ is a complex subvariety of $O$.

Next, suppose that $z\in Z$. Then, since each $Z_{j\ell}$ is an open subset of $Z_{j\ell}'$,  there exists an open neighbourhood $V$ of $z$ in $\mathbb C^n\setminus C$ and a set of holomorphic functions $h_1,\dots, h_m$ on $V$ such that
$$
Z\cap V=\{v\in V\, ;\, h_1(v)=\cdots=h_m(v)=0\}.
$$
As before we will prove that there exists an open neighbourhood $O\Subset V$ of $z$ such that $W\cap O=Z\cap O$. This will show that $W\cap O$ is a complex subvariety of $O$. 

For otherwise, we can find open neighbourhoods $O_i\Subset V$ of $z$ with $\cap_{i=1}^\infty O_i=\{z\}$ and points $x_i\in (Y\setminus Z)\cap O_i$. Let $y_i\in F^{-1}(x_i)\cap E$. We may assume without loss of generality that $\{y_i\}$ converges to a point $y\in {\rm cl}(E)$. By definition $F(y)=z\in Z'$ and so $y\in U_{j\ell}'\cap {\rm cl}(E)$ for some $j, \ell$. Since $Z'$ is a relatively open subset of $\mathbb C^n\setminus C$, there exists a (relatively) open neighbourhood $O\subset {\rm cl}(E)\cap U_{j\ell}'$ of $y$ such that $F(O)\subset Y\cap V$.
This implies that pullbacks of functions $h_1,\dots, h_m$ to $O$ by $F$ vanish on $O$ and so
$h_1,\dots, h_k$ vanish at points $x_i$ for all sufficiently large $i$. That is, such $x_i\in Z\cap V\subset Z$, a contradiction.

Combining the considered cases we obtain that $W\, (\subset\mathbb C^n)$ is a one-dimensional complex space.

Further, by the construction of $W$ there exists a domain $D\Subset U$ invariant under conjugation and containing $[a,b]$ such that $F(D)\Subset W$. We define the required complex space $X$ as an irreducible component of $W$ containing $F(D)$. 

If $F|_{U\cap\mathbb R}$ is real-valued, then by our construction, $\bar{W}=W$. Moreover, $(\ \bar{}\,\circ F)(D)=F(\bar{D})=F(D)$. Hence, $\bar{X}$ is an irreducible component of $W$ and $\bar{X}\cap X\supset F(D)$. This shows that $\bar{X}=X$ and completes
the proof of the proposition.
\end{proof}

We retain notation of the previous proposition.
\begin{Lm}\label{lem1}
Let $n: X_n\to X$ be the normalization of $X$. Then there exists a holomorphic map $F_n:D\to X_n$ such that $F=n\circ F_n$. If $F|_{U\cap \mathbb R}$ is real-valued, then there exists an antiholomorphic involution $\tau: X\to X$ such that $F_n\circ\,\bar{}=\tau\circ F_n$.
\end{Lm}
\begin{proof}
Existence of $F_n$ follows from the general fact asserting that any holomorphic map of a connected complex manifold into a reduced complex space whose image is not a subset of the locus of non-normal points can be factorized by a holomorphic map into its normalization, see, e.g., \cite[page 32]{BHPV}.

Further, $X_n$ is a connected complex manifold (because $X$ is irreducible and one-dimensional) and $n$ is biholomorphic outside a discrete subset $S\subset X_n$ such that $S':=n(S)$ is the set of singular points of $X$. If $F|_{U\cap\mathbb R}$ is real-valued, then $S'$ is invariant under involution $\,\bar{}\,$ (since $X$ is invariant under it). Therefore we can define an antiholomorphic involution $\tau(z):=(n^{-1}\circ\,\bar{}\,\circ n)(z)$ for $z\in X_n\setminus S$. Since $\tau$ maps relatively compact subsets of $X_n$ to relatively compact (because $n$ and $\,\bar{}\,$ are proper maps) and $X_n$ (being a one-dimensional Stein manifold) admits a proper holomorphic embedding into $\mathbb C^3$, by Riemann's theorem on removable singularities, see, e.g., \cite{GR}, $\tau$ is extendable to an antiholomorphic involution of $X$ (denoted by the same symbol). Finally, property  
$F_n\circ\,\bar{}=\tau\circ F_n$ follows from the property $F\circ\,\bar{}=\,\bar{}\,\circ F$.
\end{proof}

From now on we assume that $F|_{U\cap\mathbb R}$ is real-valued.

Let $p:\widetilde X_n\to X_n$ be the universal covering of $X_n$. We fix a point $o\in p^{-1}(F_n(a))\subset\widetilde X_n$.  Since $D\subset\mathbb C$ is contractible, by the covering homotopy theorem, see, e.g., \cite{Hu}, there exists a unique holomorphic map $\tilde F_n:D\to\widetilde X_n$ such that $F_n=p\circ\tilde F_n$ and $\tilde F_n(a)=o$. Also $\tau: X_n\to X_n$ induces an involutive isomorphism $\tau_*$ of the fundamental group $\pi_1(X_n,F_n(a))$ with the base point $F_n(a)$. Thus by the covering homotopy theorem, there exists a unique antiholomorphic involution $\tilde\tau:\widetilde X_n\to\widetilde X_n$ such that $\tilde\tau(o)=o$ and $p\circ\tilde\tau=\tau\circ p$. Then we have
$$
p\circ\tilde\tau\circ\tilde F_n=\tau\circ F_n=F_n\circ\,\bar{}=p\circ\tilde F_n\circ\,\bar{}\quad\text{and}\quad (\tilde\tau\circ\tilde F_n)(a)=
(\tilde F_n\circ\,\bar{}\,)(a)=o.
$$
From here by the covering homotopy theorem applied to the map $\tau\circ F_n$ we obtain that
$$
\tilde\tau\circ\tilde F_n=\tilde F_n\circ\,\bar{}.
$$
In particular, 
\[
\tilde\tau(z)=z\quad\text{ for each}\quad z\in\tilde F_n(\mathbb R\cap D).
\] 

Let $x\in [a,b]$ be such that $\tilde F_n(x)\ne 0$. 
Since there are nonconstant holomorphic functions on $X$ (and therefore on $X_n$), by the Riemann  mapping theorem $\widetilde X_n$ is biholomorphic to the open unit disk $\mathbb D\subset\mathbb C$. We can choose this biholomorphism such that it sends $o$ to $0\in\mathbb C$ and (making additionally a rotation of $\mathbb D$) $\tilde F_n(x)$ to a point $y\in (0,1)\subset\mathbb D$. In what follows we identify $\widetilde X_n$ with $\mathbb D$ under this biholomorphism. Then $\tilde\tau :\mathbb D\to\mathbb D$ is an antiholomorphic involution preserving points $0$ and $y$. Now, the composite $\,\bar{}\,\circ\tilde\tau:\mathbb D\to\mathbb D$ is a biholomorphic map of $\mathbb D$ preserving $0$ and $y$. Since each such map has a form $z\mapsto e^{i\varphi}\cdot\frac{z-a}{1-\bar{a}z}$ for some $\varphi\in\mathbb R$ and $a\in\mathbb D$, we obtain that $\,\bar{}\,\circ\tilde\tau={\rm id}$,
i.e., 
\[
\tilde\tau \ \text{coincides with the conjugation map}\ \,\bar{}\,.
\]
This implies that $\tilde F_n(\mathbb R\cap D)\subset\mathbb R\cap\mathbb D\, (=(-1,1))$.

\medskip

Let $K:=F_n([a,b])$. Then $K$ is a connected one-dimensional subanalytic subset of $X_n$. In particular, $K$ is arcwise connected and locally simply connected and the fundamental group $\pi_1(K,F_n(a))$ and covering spaces of $K$ are well defined. We will show that $K$ {\em is analytically isomorphic either to a compact interval in $\mathbb R$ or to $\mathbb S$}.

\medskip

Let $\gamma: [0,1]\to X_n$ be a path with origin at $F_n(a)$. Then there exists a unique  path
$\tilde\gamma: [0,1]\to\mathbb D$ such that $\tilde\gamma(0)=0$ and $p\circ\tilde\gamma=\gamma$. Note that $\tilde\tau\circ\tilde\gamma$ is the unique path with the origin at $0$ which covers the path $\tau\circ\gamma: [0,1]\to X_n$. In particular, if $\gamma([0,1])\subset K$, then
$\tilde\tau\circ\tilde\gamma=\tilde\gamma$, that is, $\tilde\gamma([0,1])\Subset (-1,1)$.

Next, the embedding $i: K\hookrightarrow X_n$ induces a homomorphism of fundamental groups $i_*: \pi_1(K,F_n(a))\to \pi_1(X_n,F_n(a))$. 
We will consider two cases:

\medskip

(A) $\text{Range}(i_*)=\{1\}$.

\medskip

In this case the universal covering $p:\mathbb D\to X_n$ admits a holomorphic section $s:N\to\mathbb D$ defined in an open neighbourhood $N\Subset X_n$ of $K$ (because $K$ being compact subanalytic is a deformation retract of some of its open neighbourhoods, see, e.g., \cite[Th.~2]{H}) such that $s(F_n(a))=0$. Since $\tilde F_n(\mathbb R\cap D)\subset\mathbb R\cap\mathbb D$, $s$ maps $K$ into the compact interval
$\tilde F_n([a,b])\subset\mathbb R\cap\mathbb D$.

\medskip

(B) $\text{Range}(i_*)\ne\{1\}$.

\medskip

Since $\pi_1(X_n,F_n(a))$ is a free group (because $X_n$ is a noncompact Riemann surface), by the Nielsen--Schreier theorem $\text{Range}(i_*)$ is free as well. We will show that
$\text{Range}(i_*)\cong\mathbb Z$. 

Let $\gamma: [0,1]\to K$ be a closed path representing a generator $g$ of $\text{Range}(i_*)\subset\pi_1(X_n,F_n(a))$. Then there exists a unique  path
$\tilde\gamma: [0,1]\to\mathbb D$ such that $\tilde\gamma(0)=0$, $p\circ\tilde\gamma=\gamma$ and  $\tilde\gamma([0,1])\subset (-1,1)$. 
Next, recall that the deck transformation group $\pi_1(X_n,F_n(a))$ of the covering $p:\mathbb D\to X_n$ acts discretely on $\mathbb D$ by biholomorphic transformations. From the definition of this action we obtain that $\tilde\gamma(1)=g(0)\in (-1,1)$. The above argument applied to a path representing $g^n$ shows that
$g^n(0)\in (-1,1)$ for all $n\in\mathbb Z$. Suppose that
\[
g(z)= e^{i\varphi}\cdot\frac{z-a}{1-\bar{a}z},\quad z\in\mathbb D,\quad\text{ for some}\ \varphi\in\mathbb R\ \text{ and}\ a\in\mathbb D.
\]
 Then $g^{-1}(0)=a\in (-1,1)$ and $g(0)=-e^{i\varphi}a\in (-1,1)$. This implies that $e^{i\varphi}=1$ and $a$ is real (we used here that the action of $\pi_1(X_n,F_n(a))$ on $\mathbb D$ is free), and so $g$ maps $(-1,1)$ to itself. The group 
$G:=\{g^n\, ; n\in\mathbb Z\}$ acts discretely on $(-1,1)$ and the quotient by this action is a circle $\mathbb S$. This shows that $\text{Range}(i_*)=G\cong\mathbb Z$. (For otherwise, $\text{Range}(i_*)$ contains at least two generators and so the free group $\mathbb F_2$ acts discretely by M\"{o}bius transformations on $(-1,1)$, that is, $\mathbb S$ is an infinite unbranched covering of an analytic manifold, a contradiction.) Next, $K$ is arcwise connected, and hence $p((-1,1))=K$. On the other hand, the quotient space $X_G$ of $\mathbb D$ by the action of $G$ is an unbranched covering of $X_n$ corresponding to subgroup $G\subset\pi_1(X_n,F_n(a))$. By $p_2:\mathbb D\to X_G$ and $p_1:X_G\to X_n$ we denote the covering maps so that $p=p_1\circ p_2$.
Then $p_1: p_2((-1,1))\to K$ is a finite unbranched covering and $(p_1)_*(\pi_1(X_G,0))=G$. This implies that $p_1: p_2((-1,1))\to K$ has degree one, i.e., $p_1$ determines an analytic isomorphism between $p_2((-1,1))\cong\mathbb S$ and $K$.

Let us show that {\em there exists a biholomorphic map between $X_G$ and an annulus $A\subset\mathbb C$ mapping $p_2((-1,1))$ onto $\mathbb S$}.

\medskip

Consider the biholomorphic map $w:\mathbb D\to\mathbb H_+:=\{z\in\mathbb C\, ;\, {\rm Im}(z)>0\}$, $w(z)=\frac{i(z+1)}{1-z}$, $z\in\mathbb D$, mapping $(-1,1)$ onto the positive ray of the $y$-axis. Identifying $\mathbb D$ with $\mathbb H_+$ by $w$ we easily compute that the action of $G$ on $\mathbb H_+$ is given by maps $z\mapsto \left(\frac{1-a}{1+a}\right)^n\cdot z$, $n\in\mathbb Z$, $z\in\mathbb H_+$, where $a$ is the same as in the definition of $g$ above. Let 
\[
E(z):=\exp\left(\frac{2\pi i\cdot ({\rm Log}z-\frac{\pi i}{2} )}{\ln \left(\frac{1-a}{1+a}\right)}\right),\quad z\in\mathbb H_+,
\]
where ${\rm Log}: \mathbb C\setminus (-\infty, 0]\to\mathbb C$ is the principal branch of the logarithmic function. Then $E$ is invariant with respect to the action of $G$ on $\mathbb H_+$ and maps $\mathbb H_+$ onto the annulus 
\[
A:=\left\{z\in\mathbb C\, ;\, -\frac{\pi^2}{\left|\ln \left(\frac{1-a}{1+a}\right)\right|}<\ln |z|<\frac{\pi^2}{\left|\ln \left(\frac{1-a}{1+a}\right)\right|}\right\}
\]
and the positive ray of the imaginary axis onto $\mathbb S$. This shows that $X_G\cong A$ as required.

So we may identify $X_G$ with $A$ and $p_2((-1,1))$ with $\mathbb S\subset A$. Now, there exists an analytic section $\tilde s: K\to A$ of the covering $p_1: A\to X_n$. Since $K$ is a deformation retract of its neighbourhood, say, $N$, section $\tilde s$ is a restriction of a holomorphic section $s : N\to A$ of the covering.

\subsection{} 
\begin{proof}[Proof of Theorem \ref{te1}]
We retain notation of subsection 2.1. Without loss of generality we may assume that $F=(f_1,\dots, f_n)$, where $f_1$ is nonconstant, and $I=[a,b]$. We will consider separately cases (A) and (B) of the previous subsection.

\medskip

(A) $\text{Range}(i_*)=\{1\}$.

\medskip Diminishing domain $D\Subset U$ containing $[a,b]$, if necessary, we may assume without loss of generality that $F_n(D)\Subset N$. Consider the holomorphic map $s\circ F_N: D\to\mathbb D$ such that $(s\circ F_n)([a,b])\subset (-1,1)$. We set
\[
h(z):=\exp\bigl(i\cdot(s\circ F_n)(z)\bigr),\quad z\in D;\qquad G(z):=(n\circ p)\bigl(-i\cdot {\rm Log}(z)\bigr),\quad z\in h(D).
\]
Then $G\circ h=n\circ p\circ s\circ F_n=n\circ F_n=F$ on $D$, $(h([a,b]))$ is a proper subset of $\mathbb S$, and $G$ is injective outside the set
$G^{-1}(S'\cap F(D))$, where $S'$ is the set of singular points of the complex space $X$. Since by our construction $F(D)\Subset X$, set $S'\cap F(D)$ is finite and therefore set $G^{-1}(S'\cap F(D))$ is finite as well because the normalization map $n: X_n\to X_n$ is finite and 
$p: s(N)\to N$ is biholomorphic.

This completes the proof of the first part of the theorem in this case.

\medskip

(B) $\text{Range}(i_*)\ne\{1\}$.

\medskip

As before, we may assume without loss of generality that $F_n(D)\Subset N$. Consider the holomorphic map $s\circ F_N: D\to A$ such that 
$(s\circ F_n)([a,b])=\mathbb S$. We set
\[
h(z):=(s\circ F_n)(z), \quad z\in D;\qquad G(z):=(n\circ p_1)(z),\quad z\in h(D).
\]
One easily shows that these maps satisfy conditions of the theorem.

\medskip

To prove the second part of the theorem, suppose that $F|_{D_1}=G_1\circ h_1$ for some holomorphic maps $h_1:D_1\to\mathbb C$ of a domain $D_1\Subset U$ containing $[a,b]$ and $G_1:h_1(D_1)\to\mathbb C^n$. Then $G_1$ maps open set $h_1(D\cap D_1)$ into the complex space $X$. By the definition the holomorphic map $G^{-1}: F(D)\setminus S'\to h(D)\setminus S$  is well-defined (here $S'$ is the set of singular points of $X$ and $S:=n^{-1}(S')$). Consider the holomorphic map $G^{-1}\circ G_1 : h_{1}(D\cap D_1)\setminus G_1^{-1}(S')\to h(D)$. Since $G_1^{-1}(S')$ is a discrete subset of the domain $h_1(D_1)$ and $h(D)\Subset\mathbb C$, by the Riemann theorem on removal singularities $G^{-1}\circ G_1$ is extended to a holomorphic map $q: h_{1}(D\cap D_1)\to h(D)$. We have
\[
G\circ q\circ h_1=G\circ h\qquad\text{on}\qquad D\cap D_1
\]
and $G$ is injective outside a finite subset of $h(D)$. Thus $q\circ h_1=h$ outside a finite subset of $D\cap D_1$; hence the same identity is valid on $D\cap D_1$ as well.

The proof of the theorem is complete.
\end{proof}

\sect{Proof of Theorem \ref{compos}}
\subsection{} First, we prove the theorem for ${\mathcal F}=\{f_1,\dots, f_n\}$ being a finite set of real analytic functions defined in an open neighbourhood of $[a,b]$. Then all $f_i$ are restrictions to $\mathbb R$ of holomorphic functions defined in an open neighbourhood of $[a,b]$ and we may assume that at least one of $f_i$ is nonconstant (for otherwise, the statement of the theorem is obvious). According to Theorem \ref{te1}, there exist an analytic map $h: (a-\varepsilon,b+\varepsilon)\to\mathbb S$, where $\varepsilon>0$ is sufficiently small, and analytic functions $g_1,\dots, g_n :V\to\mathbb R$, where $V\subset\mathbb S$ is an open neighbourhood of $h\bigl((a-\varepsilon,b+\varepsilon)\bigr)$, such that 
$f_i=g_i\circ h$ for all $1\le i\le n$ and the map $G=(g_1,\dots, g_n): V\to\mathbb R^n$
is injective outside a finite subset of $V$.

To prove the result it suffices to show that condition (2a) of Theorem 2  implies CC for ${\mathcal F}$. (The converse to this statement is a particular case of Theorem 1.) Condition (2a) states that the path $F:[a,b]\to\Gamma$ is closed and represents an element of the commutator subgroup $[\pi_1(\Gamma),\pi_1(\Gamma)]\subset\pi_1(\Gamma)$; here $F=(f_1,\dots, f_n)$ and $\Gamma:=F([a,b])$. We will consider two cases.

\medskip

(1) $h(a)=h(b)$. 

\medskip

Then for the corresponding induced maps in homology we have
\[
0=F_*(\gamma)=(G_*\circ \hat h_*)(\gamma);
\]
here $\gamma$ is a generator of $H_1\bigl([a,b]/(a=b)\bigr)\cong\mathbb Z$, where the space $[a,b]/(a= b)$ is obtained from $[a,b]$ by gluing together points $a$ and $b$ (i.e., it is homeomorphic to $\mathbb S$), and $\hat h: [a,b]/(a=b)\to\mathbb S$ is the map induced by $h|_{[a,b]}$.
By definition $\hat h_*(\gamma)\in H_1(\mathbb S)$ and so $\hat h_*(\gamma)=n\cdot e$ for some $n\in\mathbb Z$, where $e$ is a generator of $H_1(\mathbb S)\cong\mathbb Z$. Hence $n\cdot G_*(e)=0$.

Suppose $n\ne 0$; then $h: [a,b]\to\mathbb S$ is surjective. Since $\Gamma$ is homotopic to a finite one-dimensional CW complex, $H_1(\Gamma)$ is a free abelian group $\mathbb Z^d$, $d\in\mathbb Z_+$. Hence, condition $n\cdot G_*(e)=0$ implies that $G_*(e)=0$. Therefore the image of the path ${\rm id}:\mathbb S\to\mathbb S$ under $G$ represents an element of the commutator subgroup $[\pi_1(\Gamma),\pi_1(\Gamma)]$. This implies that for the map $G\circ {\rm id}:\mathbb S\to\Gamma$ there exists an open arc $\ell\subset\Gamma$ such that $(G\circ {\rm id})^{-1}(\ell)$ is disjoint union of at least two open subarcs of $\mathbb S$. However, $G\circ {\rm id}$ is injective outside a finite subset of $\mathbb S$; a contradiction showing that $\hat h_*(\gamma)=0$.

\noindent Next, $\pi_1(\mathbb S)\cong H_1(\mathbb S)\cong\mathbb Z$ so the latter condition implies that the induced by $h$ path $\hat h: [a,b]/(a=b)\to\mathbb S$ is contractible. Let $p:\mathbb R\to\mathbb S$, $p(x)=e^{2\pi i x}$, $x\in\mathbb R$, be the universal covering of $\mathbb S$. Then by the covering homotopy theorem applied to $h$, there exists an analytic path $\tilde h : (a-\varepsilon,b+\varepsilon)\to\mathbb R$ with $\tilde h(a)=\tilde h(b)$ such that $h=p\circ\tilde h$. We set $\tilde F:=G\circ p$. Then
$F=\tilde F\circ \tilde h|_{[a,b]}$, i.e., the family ${\mathcal F}$ satisfies CC.

\medskip

(2) $h(a)\ne h(b)$.

\medskip

Since $h$ is analytic in an open neighbourhood of $[a,b]$, the set $\{b'\in [a,b]\, ;\, h(b')=h(a)\}$ is finite. By $\tilde b$ we denote its maximal element. Then $h([\tilde b,b])\subset\mathbb S$ is a proper closed arc. 

First, assume that $\tilde b\ne a$. Then $h: [a,b]\to\mathbb S$ is a surjection.

Let $K$ be a compact space obtained from $[a,b]$ by gluing together points $a$, $b$ and $\tilde b$ (a figure-eight space). By $q_1: [a,b]\to K$ we denote the quotient map. Since $F(a)=F(\tilde b)=F(b)$ there exists a continuous map $F_1: K\to\Gamma$ such that $F=F_1\circ q_1$. Also, $q_1$ is the composite of the quotient map $q: [a,b]\to [a,b]/(a=b)\, (\cong\mathbb S)$ and the quotient map
$q_2: [a,b]/(a=b)\to K$ (which identifies point $q(a)=q(b)$ with $q(\tilde b)$). According to our assumption, the path $F_1\circ q_2: [a,b]/(a=b)\to\Gamma$ represents an element of $[\pi_1(\Gamma),\pi_1(\Gamma)]$. Then for the induced maps in homology we have
$((F_1)_*\circ (q_2)_*)(\gamma)=0$, where $\gamma$ is a generator of $H_1\bigl([a,b]/(a=b)\bigr)\cong\mathbb Z$. Note that $(q_2)_*(\gamma)=\delta_1+\delta_2$, where $\delta_i$, $i=1,2$, are generators of $H_1(K)\cong\mathbb Z^2$.

Further, by $L$ we denote the compact space obtained from $\mathbb S$ by gluing together points $h(a)$ and $h(\tilde b)$ (another figure-eight space). By $r:\mathbb S\to L$ we denote the quotient map. Since $G(h(a))=G(h(\tilde b))$ and $h$ is surjective, there exists a continuous surjective map $G_1:L\to\Gamma$ such that $G=G_1\circ r$. Moreover, $G_1$ is injective outside a finite subset of $L$ and, hence, it determines an injective homomorphism of homology groups $(G_1)_*: H_1(L)\to H_1(\Gamma)$. 

Now, the maps $h|_{[a,\tilde b]}$ and $h|_{[\tilde b, b]}$ determine maps $h_1: K_1\to\mathbb S$ and $h_2: K_2\to L$, where $K_1$ and $K_2$ are `circles` forming the figure-eight space $K$, such that
\[
h|_{[a,\tilde b]}=h_1\circ q_1|_{[a,\tilde b]},\qquad r\circ h|_{[\tilde b, b]}=h_2\circ q_1|_{[\tilde b, b]}.
\]
Cycles $K_1$ and $K_2$ (choosing with a proper orientation) represent elements $\delta_1$ and $\delta_2$ in $H_1(K)$. Hence we have
\[
(G_1)_*\bigl((r_*\circ (h_1)_*)(\delta_1)+(h_2)_*(\delta_2)\bigr)=(F_1)_*(\delta_1+\delta_2)=0.
\]
This implies that $(r_*\circ (h_1)_*)(\delta_1)+(h_2)_*(\delta_2)=0\in H_1(L)$. On the other hand, $(r_*\circ (h_1)_*)(\delta_1)=n_1(e_1+e_2)$ and
$(h_2)_*(\delta_2)=n_2 e_2$ for some $n_1,n_2\in\mathbb Z$, where $e_1$ and $e_2$ are generators of $H_1(L)\cong\mathbb Z^2$ corresponding to cycles (`circles`) forming $L$ chosen with appropriate orientations. Therefore, $n_1e_1+(n_1+n_2)e_2=0$, that is, $n_1=n_2=0$.

Finally, since $h([\tilde b,b])\subset\mathbb S$ is a proper closed arc, there exists a path $\tilde h: [\tilde b,b]\to\mathbb R$ such that
$h|_{[\tilde b,b]}=p\circ\tilde h$ and $p$ is one-to-one in an open neighbourhood of the interval $\tilde h([\tilde b,b])$. Let
$g: [\tilde b,b]\to \tilde h([\tilde b,b])$ be the linear function such that $g(\tilde b)=h(\tilde b)$ and $g(b)=h(b)$. Then 
$H(t)=t\cdot\tilde h+(1-t)\cdot g$, $0\le t\le 1$, is a homotopy between $\tilde h$ and $g$ fixing points $h(a)$ and $h(b)$ and such that the range of each $H(t)$ belongs to $\tilde h([\tilde b,b])$. In turn, $r\circ p\circ H$ is a homotopy between closed paths $h_2: K_2\to L$ and
$h_2':K_2\to L$, where $h_2'\circ q_1=r\circ p\circ g$. Thus, by our construction, $0=(h_2)_*(\delta_2)=(h_2')_*(\delta_2)=\pm e_2$, a contradiction.

Second, assume that $\tilde b=a$. Then $h([a,b])\subset\mathbb S$ is a proper closed arc and we apply arguments similar to the above to $[a,b]/(a=b)$ instead of $K$ (here we do not have $\delta_1$) to get a contradiction. We leave the details to the reader.

This shows that case (2) cannot occur and completes the proof of the theorem for the finite family ${\mathcal F}$.

\subsection{} Now, suppose that ${\mathcal F}=\{f_1,\dots, f_n\}$ are complex analytic in an open neighbourhood of $[a,b]$ and such that
$\widehat\Gamma=\Gamma$, where $\Gamma:=F([a,b])$, $F:=(f_1,\dots, f_n)$. We write $f_j=f_j'+if_j''$, where $f_j'$ and $f_j''$ are real and imaginary parts of $f_j$, $1\le j\le n$, and apply Theorem \ref{te1} to the family of real analytic functions $\{f_1',f_2'',\dots, f_n',f_n''\}$.
Thus there exist an analytic map $h: (a-\varepsilon,b+\varepsilon)\to\mathbb S$, where $\varepsilon>0$ is sufficiently small, and analytic functions $g_1',g_2'',\dots, g_n',g_n'' :V\to\mathbb R$, where $V\subset\mathbb S$ is an open neighbourhood of $h\bigl((a-\varepsilon,b+\varepsilon)\bigr)$, such that 
$f_j'=g_j'\circ h$, $f_j''=g_j''\circ h$ for all $1\le j\le n$ and the map $(g_1',g_1'',\dots, g_n',g_n''): V\to\mathbb R^n$
is injective outside a finite subset of $V$. We set $g_j:=g_j'+ ig_j''$, $1\le j\le n$. Then $f_j=g_j\circ h$ and the map
$G:=(g_1,\dots, g_n): V\to \mathbb C$ is injective outside a finite subset of $V$ as well. 

Further, if the path $F:[a,b]\to\Gamma$ is closed and represents an element of the commutator subgroup $[\pi_1(\Gamma),\pi_1(\Gamma)]\subset\pi_1(\Gamma)$, then repeating word-for-word the arguments of the previous subsection we obtain that ${\mathcal F}$ satisfies CC. The converse to this statement follows from Theorem 1.

This completes the proof of the theorem for finite families of analytic functions.

\subsection{} Let us consider the general case. Suppose families ${\mathcal F}_1=\{f_1,\dots, f_{n_1}\}$ and ${\mathcal F}_2=\{f_1,\dots, f_{n_1},\dots , f_{n_2}\}$  of nonconstant analytic functions satisfy hypotheses of Theorem \ref{compos}. We set $F_{\mathcal F_i}=(f_1,\dots, f_{n_i})$, 
$i=1,2$. According to Theorem \ref{te1}, there exist analytic maps $h_{\mathcal F_i}: (a-\varepsilon, b+\varepsilon)\to\mathbb S$ for some $\varepsilon>0$, $G_{\mathcal F_i}: V_{\mathcal F_i}\to \mathbb C^{n_i}$ defined in some open neighbourhoods $V_{\mathcal F_i}\subset\mathbb S$ of $h_{\mathcal F_i}\bigl((a-\varepsilon,b+\varepsilon)\bigr)$, $i=1,2$, and an analytic map $q_{\mathcal F_1}^{\mathcal F_2}: V_{\mathcal F_2}\to V_{\mathcal F_1}$ 
such that 

\medskip

\noindent {\em $F_{\mathcal F_i}=G_{\mathcal F_i}\circ h_{\mathcal F_i}$ and $G_{\mathcal F_i}$ are injective outside finite subsets of  $V_{\mathcal F_i}$, $i=1,2$;}

\medskip  

\noindent {\em
If  $G_{\mathcal F_1}^{\mathcal F_2}=p_{\mathcal F_1}^{\mathcal F_2}\circ G_{\mathcal F_2}$, where  $p_{\mathcal F_1}^{\mathcal F_2}:\mathbb C^{n_2}\to\mathbb C^{n_1}$ is the natural projection onto the first
$n_1$ coordinates, then $h_{\mathcal F_1}=q_{\mathcal F_1}^{\mathcal F_2}\circ h_{\mathcal F_2}$ on  $(a-\varepsilon, b+\varepsilon)$ and  $G_{\mathcal F_1}^{\mathcal F_2}=G_{\mathcal F_1}\circ q_{\mathcal F_1}^{\mathcal F_2}$ on $V_{\mathcal F_2}$.}

\medskip

Making appropriate unitary transformations of $\mathbb C$, we may assume without loss of generality that $h_{\mathcal F_i}(a)=1$, $i=1,2$. Now  results of subsections 3.1--3.2 imply that there exist analytic maps $\tilde h_{\mathcal F_i}: (a-\varepsilon, b+\varepsilon)\to\mathbb R$ such that $\tilde h_{\mathcal F_i}(a)=\tilde h_{\mathcal F_i}(b)=0$ and $h_{\mathcal F_i}=p\circ\tilde h_{\mathcal F_i}$, $i=1,2$. Also, by the covering homotopy theorem, there exists an analytic map $\tilde q_{\mathcal F_1}^{\mathcal F_2}: \tilde V_{\mathcal F_2}\to\tilde V_{\mathcal F_1}$, where $\tilde V_{\mathcal F_i}$ are some open neighbourhoods of $\tilde h_{\mathcal F_i}\bigl( (a-\varepsilon, b+\varepsilon)\bigr)$, $i=1,2$, such that $\tilde q_{\mathcal F_1}^{\mathcal F_2}(0)=0$ and $p\circ \tilde q_{\mathcal F_1}^{\mathcal F_2}=q_{\mathcal F_1}^{\mathcal F_2}\circ p$.
We set $\tilde F_{\mathcal F_i}:= G_{\mathcal F_i}\circ p$, $i=1,2$, and $\tilde F_{\mathcal F_1}^{\mathcal F_2}:=G_{\mathcal F_1}^{\mathcal F_2}\circ p$. Then
\begin{equation}\label{e3.1}
\begin{array}{c}
\displaystyle
F_{\mathcal F_i}=\tilde F_{\mathcal F_i}\circ\tilde h_{\mathcal F_i}\quad\text{on}\quad  (a-\varepsilon, b+\varepsilon),\qquad \tilde h_{\mathcal F_i}(a)=\tilde h_{\mathcal F_i}(b)=0\qquad\text{and}\\
\\
\displaystyle
\tilde F_{\mathcal F_1}^{\mathcal F_2}=\tilde F_{\mathcal F_2}\circ \tilde q_{\mathcal F_1}^{\mathcal F_2}\quad\text{on}\quad \tilde V_{\mathcal F_2},\qquad \tilde h_{\mathcal F_1}=\tilde q_{\mathcal F_1}^{\mathcal F_2}\circ \tilde h_{\mathcal F_2}\quad\text{on}\quad (a-\varepsilon, b+\varepsilon).
\end{array}
\end{equation}

\medskip

Now, suppose that $\mathcal F$ is an infinite family of complex analytic functions defined in open neighbourhoods of the compact interval $[a,b]$ whose finite subfamilies satisfy condition (2a) of Theorem 2. To prove the required result it suffices to show that ${\mathcal F}$ satisfies CC. Without loss of generality we may assume that $\mathcal F$ consists of nonconstant functions. For a finite subset $\mathcal F'\subset\mathcal F$ by 
\[
C_{\mathcal F'}:=\{x\in [a,b]\, ;\, \tilde h_{\mathcal F'}'(x)=0\}
\]
we denote the set of critical points of $\tilde h_{\mathcal F'}$ counted with their multiplicities and by $L_{\mathcal F'}$ the locus of $C_{\mathcal F'}$. Then all $C_{\mathcal F'}<\infty$ and all $L_{\mathcal F'}$ are finite subsets of $[a,b]$. Also, according to \eqref{e3.1} we have
\[
C_{\mathcal F_2}\le C_{\mathcal F_1}\quad\text{and}\quad L_{\mathcal F_2}\subset L_{\mathcal F_1}\quad\text{for}\quad \mathcal F_1\subset\mathcal F_2 .
\]
The set $A$ of all finite subsets of ${\mathcal F}$ is a directed set with the natural order: $\mathcal F_1\le\mathcal F_2$ iff
$\mathcal F_1\subset\mathcal F_2$. Then $\{x_{\alpha}:=(C_{\alpha}, L_{\alpha})\, ;\, C_\alpha\in\mathbb Z_+,\, L_\alpha\subset [a,b]\}_{\alpha\in A}$ is a nonincreasing net. Equipping the set of all compact subsets of $[a,b]$ with the Hausdorff metric (so that it becomes a compact metric space) and $\mathbb Z_+$ with the topology induced from $\mathbb R$ and using the fact that $\{x_\alpha\}_{\alpha\in A}$ is nonincreasing, we find a convergent subnet $\{x_{\alpha_\beta}\}_{\beta\in B}$ of $\{x_\alpha\}_{\alpha\in A}$. Since $\{x_{\alpha_\beta}\}$ is nonincreasing and all $C_{\alpha_{\beta}}<\infty$ and all $L_{\alpha_{\beta}}$ are finite subsets of $[a,b]$, there exists $\beta_0\in B$ such that
\begin{equation}\label{e3.2}
x_{\alpha_\beta}=x_{\alpha_{\beta_0}}\quad\text{for all}\quad \beta\ge\beta_0 .
\end{equation}
We set
\[
x_{\alpha_{\beta_0}}:=(C_{\mathcal F}, L_{\mathcal F}).
\]
Using \eqref{e3.2} and \eqref{e3.1} we obtain 
\[
C_\alpha=C_{\mathcal F}\quad\text{and}\quad L_\alpha=L_{\mathcal F}\quad\text{for all}\quad \alpha\ge \alpha_{\beta_0}.
\]
This implies that nonconstant real analytic functions $\tilde q^{\alpha}_{\alpha_{\beta_0}}$ do not have critical points on $\tilde h_\alpha([a,b])$. In particular, they are invertible on $\tilde h_\alpha([a,b])$. Hence,
\[
\tilde h_\alpha=(\tilde q^{\alpha}_{\alpha_{\beta_0}})^{-1}\circ \tilde h_{\alpha_{\beta_0}}\quad\text{and}\quad F_\alpha=\left(\tilde F_\alpha\circ (\tilde q^{\alpha}_{\alpha_{\beta_0}})^{-1}\right)\circ \tilde h_{\alpha_{\beta_0}}\quad
\text{for all}\quad \alpha\ge \alpha_{\beta_0}.
\]
Thus the family ${\mathcal F}$ satisfies CC with the composition factor $\tilde h_{\alpha_{\beta_0}}$.

The proof of the theorem is complete.

\end{document}